\pdfoutput=1
\RequirePackage{ifpdf}
\ifpdf % We~are running pdfTeX in pdf mode
\documentclass[pdftex]{sigma}
\else
\documentclass{sigma}
\fi

\numberwithin{equation}{section}

\newtheorem{Theorem}{Theorem}[section]
\newtheorem{Corollary}[Theorem]{Corollary}
\newtheorem{Lemma}[Theorem]{Lemma}
\newtheorem{Proposition}[Theorem]{Proposition}
 { \theoremstyle{definition}
\newtheorem{Definition}[Theorem]{Definition}
\newtheorem{Example}[Theorem]{Example}
\newtheorem{Remark}[Theorem]{Remark} }

\usepackage[vcentermath]{youngtab}

\def\la{\lambda}
\newcommand{\inv}{^{-1}}

\newcommand{\C}{\mathbb{C}}

\begin{document}
\allowdisplaybreaks

\newcommand{\arXivNumber}{2003.08616}

\renewcommand{\PaperNumber}{070}

\FirstPageHeading

\ShortArticleName{Singularities of Schubert Varieties within a Right Cell}

\ArticleName{Singularities of Schubert Varieties within a Right Cell}

\Author{Martina LANINI~$^{\rm a}$ and Peter J.~MCNAMARA~$^{\rm b}$}

\AuthorNameForHeading{M.~Lanini and P.J.~McNamara}

\Address{$^{\rm a)}$~Department of Mathematics, University of Rome ``Tor Vergata'', Italy}
\EmailD{\href{lanini@mat.uniroma2.it}{lanini@mat.uniroma2.it}}
\URLaddressD{\url{https://sites.google.com/site/martinalanini5/}}

\Address{$^{\rm b)}$~School of Mathematics and Statistics, The University of Melbourne, Australia}
\EmailD{\href{maths@petermc.net}{maths@petermc.net}}
\URLaddressD{\url{http://petermc.net/maths/}}

\ArticleDates{Received December 10, 2020, in final form July 06, 2021; Published online July 19, 2021}

\Abstract{We describe an algorithm which pattern embeds, in the sense of Woo--Yong, any Bruhat interval of a symmetric group into an interval whose extremes lie in the same right Kazhdan--Lusztig cell. This apparently harmless fact has applications in finding examples of reducible associated varieties of $\mathfrak{sl}_n$-highest weight modules, as well as in the study of $W$-graphs for symmetric groups, and in comparing various bases of irreducible representations of the symmetric group or its Hecke algebra. For example, we are able to systematically produce many negative answers to a question from the 1980s of Borho--Brylinski and Joseph, which had been settled by Williamson via computer calculations only in 2014.}


\Classification{14M15; 20B30; 32C38; 20C08}

\section{Introduction}

Let $\mathcal{F}l_n$ be the variety of complete flags in $\C^n$. For each permutation $w\in S_n$, the Schubert variety $X_w$ is a closed subvariety of $\mathcal{F}l_n$ consisting of flags which intersect a fixed flag in a~particular way prescribed by $w$ -- see~(\ref{Eqn_SchubertVar}) for the precise definition. The Bruhat order on $S_n$ is defined by $x\leq y$ if $X_x\subseteq X_y$.

Right, left and two-sided cells were introduced by Kazhdan and Lusztig in their seminal paper~\cite{KL79} in order to construct representations of the Hecke algebra associated to a Coxeter group~$W$. In general, cells are equivalence classes on~$W$ defined using the Kazhdan--Lusztig basis of the Hecke algebra. When $W$ is a symmetric group, there is a simple combinatorial description in terms of the Robinson--Schensted correspondence: namely that two permutations belong to the same right, left, two-sided cell respectively if and only if their~$P$ symbols coincide, $Q$ symbols coincide, their $P$ (or, equivalently, $Q$) symbols have the same shape.

The key result of this paper is the following:
\begin{Theorem}\label{Thm_Main}Let $x,y\in S_n$ with $x\leq y$. Then, there exist $N\geq n$, and two permutations $v,w\in S_N$ with $v\leq w$ such that
\begin{itemize}\itemsep=0pt
\item $v$ and $w$ belong to the same right cell;
\item the singularity of $X_w$ at $v$ is smoothly equivalent to the singularity of $X_y$ at $x$.
\end{itemize}
\end{Theorem}
The permutations $v$, $w$ of the above theorem are obtained from $x$ and $y$ by an explicit algorithm, described in Section~\ref{sec:combinatorics}. Each step of the algorithm produces an interval pattern embedding (see Definition \ref{defn_patternavoiding}), which, by results of Woo--Yong \cite{wooyong}, induces smooth equivalences.

As a consequence, any possible behaviour of invariants of Schubert variety singularities (e.g., Kazhdan--Lusztig polynomials, their positive characteristic analogues, decomposition numbers of perverse sheaves, etc.) can appear within a single right cell. An analogous statement holds for left cells. Any of these statements imply the same result for two-sided cells.

Historically there had been some hope that the singularity of a Schubert variety $X_w$ at a~point~$v$ with~$v$ and~$w$ in the same cell would be better behaved than general singularities of Schubert varieties. Our work shows this not to be the case. An example of this hope is Borho--Brylinski and Joseph's conjecture from the 1980s that the associated variety of an $\mathfrak{sl}_n$-module would be always irreducible \cite{BB,J}. Such a question was settled only in 2014 by Williamson in \cite{williamson}. For the moment, we just want to mention that the answer provided in~\cite{williamson} is obtained by using Howlett and Nguyen's software \cite{HN} for computing $W$-graphs in {\tt MAGMA} \cite{BCP}. Thanks to our main result, this is now a consequence of the known reducibility of type~A characteristic varieties. We discuss this in greater detail in Section~\ref{Sec_BBJ}.

In Sections~\ref{Sec_01} and \ref{sec:specht} two further applications of Theorem~\ref{Thm_Main} are presented. The first one deals with the study of edge weights in $W$-graphs for symmetric groups and the so-called 0-1 conjecture, which would have implied that all the weights are either 0 or 1. Thanks to \cite{mclarnanwarrington} the conjecture was already known to be false, but we obtain in Section~\ref{Sec_01} a counterexample which does not come from computer calculations. The last application we discuss is about the comparison of several Specht module bases -- the Springer basis, the Goldie rank basis, the Kazhdan--Lusztig basis and the $p$-Kazhdan--Lusztig bases for each prime~$p$ (see Section~\ref{Sec_Applications} for more details). Our approach allows us to transfer known examples of singularities of Schubert varieties into a single cell and thus exhibit many examples where these bases differ. In fact, thanks to the torsion explosion phenomenon from \cite{explosion} we know that there are infinitely many such examples, as explained in Section~\ref{sec:specht}.

\section{Robinson--Schensted correspondence}\label{sec:combinatorics}

The Robinson--Schensted correspondence is a bijection between $S_n$ and pairs of standard tab\-leaux of the same shape with $n$ boxes. If $w\in S_n$, we write $(P(w),Q(w))$ for the corresponding pair of tableaux. The $P$ symbol $P(w)$ is obtained by successively performing $n$ column insertions into the empty tableau, with numbers $w(n),w(n-1),\dots,w(1)$ in that order. The equivalence of this with the more familiar row insertion definition follows from \cite[Lemma 7.23.15]{stanley}. The $Q$ symbol $Q(w)$ is obtained by the recording the cell added at each insertion step: after adding the $i$-th cell to $P(w)$, one adds a cell with content $i$ in the same position as the new cell in $P(w)$.

Here, and in the rest of the paper, we use the one line notation for permutations: we represent $w\in S_n$ as $[w_1\dots w_i\dots w_n]$, meaning that $w_i=w(i)$ for any $1\leq i\leq n$.
\begin{Example}\label{exple:PQsymbols}If $w=[3142]\in S_4$, the (column) insertion procedure produces the following two sequences of tableaux, and hence the $P$ and $Q$ symbols:
\[
\young(2), \ \young(2,4), \ \young(12,4), \ \young(12,34)=P(w),
%\]
\qquad
%\[
\young(1), \ \young(1,2), \ \young(13,2), \ \young(13,24)=Q(w).
\]
\end{Example}

Left and right cells are equivalence classes on a Coxeter group defined in terms of Kazhdan--Lusztig polynomials by Kazhdan and Lusztig (cf.~\cite{KL79} after Theorem~1.3). We omit the original definition here, as we only deal with the symmetric group case where a combinatorial characterisation exists. The following is
\cite[Theorem~A]{ariki} or \cite[Fact~8]{gm}, and is originally from \cite[Section~5]{KL79}.
It gives a characterisation of left and right cells in the symmetric group.

\begin{Theorem}\label{thm:defnCells}
Two permutations $x$ and $y$ in $S_n$ are in the same right $($respectively left$)$ cell if and only if $P(x)=P(y)$ $($respectively $Q(x)=Q(y))$.
\end{Theorem}

Two elements $x$ and $y$ of a Coxeter group are in the same left cell if and only if $x\inv$ and $y\inv$ are in the same right cell.
For a standard tableau $T$ and an entry $s$ of $T$, we denote by $c_T(s)$ the column index of the entry. For example if $T=P(w)$ from Example~\ref{exple:PQsymbols}, then
\(c_T(1)=c_T(3)=1\) and \(c_T(2)=c_T(4)=2\).

\begin{Lemma}\label{lem:2.2}
Let $T$ be a standard tableau, and let $s$ be an entry of $T$. Let $r$ be such that no entries of $T$ lie in the interval $(r,s)$ and let $k\leq r-s$ be a positive integer. Choose $k-1$ integers $r<r_1<\cdots<r_{k-1}<s$. Let $T'$ be the tableau obtained from $T$ by column inserting $r_{k-1},\dots,r_1$. Then
\begin{enumerate}\itemsep=0pt
\item[$1)$] $c_{T'}(r_i)=i$ and $c_{T'}(s)=m$, where $m=\max\{c_T(s),k\}$,
\item[$2)$] if $s'$ is an entry of $T$ such that $s'>s$, then $c_{T'}(s')=c_T(s')$ or $c_{T'}(s')\leq m+n$, where $m$ is as before and $n$ is the number of entries in $T$ in $[s,s')$.
\end{enumerate}
\end{Lemma}
\begin{proof}
 By induction on $t$, we show that if we column insert $r_{k-1}, \dots, r_{k-t}$, then
\begin{itemize}\itemsep=0pt
\item $r_{k-t+i}$ lies in the $(i+1)$-st column for $0\leq i<t$,
\item $s$ lies in the $m$-th column, for $m=\max\{c_T(s),t+1\}$.
\end{itemize}
The base case $t=0$ is tautologically true. Now suppose we have inserted $r_{k-1}, \dots, r_{k-t+1}$, in accordance with the inductive hypothesis. After these insertions, $s$ will be placed in column $\max\{c_T(s),t\}$. When we column insert $r_{k-t}$, any $r_{k-t+i}$ bumps out $r_{k-t+i+1}$ from the $i$-th column. In particular, $r_{k-1}$ is placed in the $t$-th column and $s$ gets bumped out if and only if $\max\{c_T(s),t\}=t$, that is $c_T(s)\leq t$. This proves the first part of the lemma.

 We now prove the second statement by induction on $n$. Let $s_1<\dots<s_{n-1}$ be the elements of $(s,s')$ that lie in $T$.
 During the column insertions $s'$ can be bumped out only by the entries between $r$ and $s'$. By the previous part of this lemma, the entries $r_1, \dots,r_{k-1}, s$ lie in the first~$m$ columns of $T'$, and hence can bump $s'$ out from its box only if $c_{T}(s')\leq m$ and none of the entries $s_1, \dots, s_{n-1}$ appear in the first $m$ columns. If this is the case, $c_{T'}(s')\leq m+1$. Clearly, the entries $s_1,\dots,s_{n-1}$ can bump $s'$ out of its box only if they move. Hence, assume that $s'$ gets bumped out from its box the first time by an $s_i$, then $s'$ will always be bumped out by~$s_i$, unless it lands on a column which contains $s_j$ for some $j$. If this is the case, then $s'$ stays on that column. We conclude that $c_{T'}(s')\leq c_{T'}(s_{i})+1$, and hence the statement follows by induction.
\end{proof}

Let $x\in S_n$ and $v\in S_N$ (for some $n\leq N$). We say that the pattern of $v$ in the last $n$ positions is $x$ if it holds that $v(N-n+i)<v(N-n+j)$ if and only if $x(i)<y(j)$ for all $1\leq i<j\leq n$.

Thanks to the first part of the previous lemma, we are now able to prove the following central result.

\begin{Theorem}\label{thm:asyouwish}
Given $x,y\in S_n$, there exist $v,w\in S_N$ for some $N\geq n$ with $P(v)=P(w)$, $v(i)=w(i)$ for $i\leq N-n$ and with the pattern of $v$ and $w$ in the last $n$ positions being the permutations $x$ and $y$.
\end{Theorem}
\begin{proof}Let $k$ be the largest integer such that $P(x)$ and $P(y)$ have all entries less than or equal to $k$ in the same place. (If $P(x)=P(y)$ we set $k=n$.) We perform an induction on the value of $n-k$, the case $n-k=0$ being trivial. Define $t$ by
\[
t=\max\{c_{P(x)}(k+1), c_{P(y)}(k+1)\}-1.
\]
Let $n'=n+t$ and define $x',y'\in S_{n'}$ by{\samepage
\begin{equation}\label{eq:prime}
x'(i)=\begin{cases}
k+i & \mbox{if}\quad i\leq t,\\
x(i-t) & \mbox{if}\quad i>t \mbox{ and } x(i-t)\leq k, \\
x(i-t)+t &\mbox{if}\quad i>t \mbox{ and } x(i-t)>k,
\end{cases}
\end{equation}
with $y'$ defined similarly from $y$.}

By the description of $P$ in terms of column insertion, $P(x')$ and $P(y')$ are obtained by inserting the numbers $k+t$, $k+t-1$, \dots, $k+1$, into the tableau obtained from $P(x)$ and $P(y)$ by adding~$t$ to all entries greater than $k$. By Lemma \ref{lem:2.2}(1), $P(x')$ and $P(y')$ agree for all entries less than or equal to $k+t+1$ (by the choice of $t$).

Let $k'$ be defined from $x'$ and $y'$ analogously to the definition of $k$. Then $n'-k'\leq (n+t)-(k+t+1)<n-k$. So by induction there exists $v$ and $w$ with $P(v)=P(w)$ and the pattern of~$v$ and $w$ in the last $n'$ positions being the permutations~$x'$ and~$y'$. Since the patterns of~$x'$ and~$y'$ in the last~$n$ positions are those of~$x$ and~$y$ respectively, these~$v$ and~$w$ satisfy the conditions of the theorem, completing the proof.
\end{proof}

\begin{Example}\label{eg:ks}
If $x=[2 1 6 5 4 3 8 7]$ and $y=[62845173]$, then the proof outputs the pair $v=[895621a743cb]$ and $w=[8956a2c471b3]$, where $a=10$, $b=11$, $c=12$. This example involves two invocations of Lemma~\ref{lem:2.2}.
\end{Example}

We choose this example because the singularity of $X_y$ at $x$ is the Kashiwara--Saito singularity \cite[Example~8.3.1]{kashiwarasaito} and to point out that the pair $v,w\in S_{12}$ that we obtain is different from the permutations chosen in \cite[Section~3.5]{williamson}. The relevance of this example will be discussed in more detail in Section~\ref{Sec_BBJ}.

\begin{Proposition}\label{prop:bound}
In Theorem {\rm \ref{thm:asyouwish}} above, we can always take $N\leq n(n+1)/2$.
\end{Proposition}

\begin{proof}
The inductive proof of Theorem \ref{Thm_Main} gives an algorithm for constructing $u$ and $v$ from~$x$ and~$y$. It produces a sequence of permutations $\{x_i\}$, $\{y_i\}$ where $x_0=x$, $y_0=y$ and $x_{i+1}=x_i'$ and $y_{i+1}=y_i'$ are defined as in~(\ref{eq:prime}) where $k_i+1$ is the minimal entry which appears in a different place in $P(x_i)$ and $P(y_i)$, and $t_i=\max\{c_{P(x_i)}(k_i+1),c_{P(y_i)}(k_i+1)\}-1$. Let $n_i=n+\sum_{j=0}^{i-1}t_j$ be the index of the symmetric group~$x_i$ and~$y_i$ lie in.

We prove now by induction on $i\geq 1$ that for any $l>k_{i-1}+t_{i-1}$
Let $i=1$. By Lemma~\ref{lem:2.2}(2), either $c_{P(x_1)}(l)=c_{P(x)}(l-t_0)$ or $c_{P(x_1)}(l)\leq (t_0+1)+l-(k_0+1+t_0)$. In the first case,
the thesis follows from $c_{P(x)}(l-t_0)\leq l-t_0=l-n_1+n$, as $P(x)$ is standard. As for the second case, we just notice that $k_0\geq t_0$, and hence also in this case we have $c_{P(x_1)}(l)\leq l-t_0$.

The induction step is proven analogously: by Lemma \ref{lem:2.2}(2) we can distinguish the two cases $c_{P(x_i)}(l)=c_{P(x_{i-1})}(l-t_{i-1})$ and $c_{P(x_i)}(l-t_{i-1})\leq (t_{i-1}+1)+l-(k_{i-1}+1+t_{i-1})=l-k_{i-1}$. In the first case, the thesis follows by induction, while for the second case it is enough to observe that $k_{i-1}\geq n_i-n$, which follows inductively from $k_i\geq k_{i-1}+t_{i-1}$.

Clearly, the same upper bound is obtained for $c_{P(y_i)}(l)$. Since $k_{i}+1> k_{i-1}+t_{i-1}$, we get
\[
c_{P(x_i)}(k_i+1), c_{P(y_i)}(k_i+1)\leq k_i+1-n_i+n.
\]
Therefore $t_i\leq n-(n_i-k_i)$. We have $N=n+\sum_i t_i$. The proof of Theorem \ref{thm:asyouwish} shows that the sequence $n_i-k_i$ is a strictly decreasing sequence of positive integers and that we have at most $n-1$ iterations. Therefore $N\leq n+\sum_{j=1}^{n-1}j=n(n+1)/2$.
\end{proof}

\section{Interval pattern embeddings}
Pattern avoidance has been applied to investigate properties of Schubert varieties for several decades. The first, and arguably most famous, instance of this is Lakshmibai--Sandhya's criterion for smoothness~\cite{lakshmibaisandhya}. We refer the interested reader to~\cite{patternavoidance} for a nice survey on the power of pattern avoidance techniques and generalisations.

Here we exploit Woo--Yong's approach of interval pattern embeddings, which gives a sufficient condition for when two singularities in different Schubert varieties are smoothly equivalent.

We first recall the notion of pattern embedding:

\begin{Definition}[{\cite[Section~2.1]{wooyong}}]\label{defn_patternavoiding} Let $2\leq n\leq m$.
\begin{enumerate}\itemsep=0pt
\item Let $y\in S_n$ and $w\in S_m$. We say that $y$ embeds into $w$ if there are indices $1\leq \varphi_1< \dots< \varphi_n\leq m$ such that $w(\varphi_j)<w(\varphi_{k})$ if and only if $y(j)<y(k)$ for any $1\leq j< k\leq n$. In this case we say that the set $\Phi=\{\varphi_1, \dots, \varphi_n\}$ is an embedding of $y$ into $w$.
\item Let $[x,y]$ and $[v,w]$ be two intervals in the Bruhat orders on $S_n$ and $S_m$. We say that $[x,y]$ embeds into $[v,w]$ if there is a common embedding $\Phi=\{\varphi_1, \dots, \varphi_n\}$ of $x$ into $v$ and of $y$ into $w$, where the entries of $v$ and $w$ agree outside of $\Phi$, and $[x,y]\simeq [v,w]$ as posets.
\end{enumerate}
\end{Definition}

Now we briefly recall the definitions of Schubert and Richardson varieties.

For $v_1, \dots, v_j\in \mathbb{C}^n$, we denote by $\langle v_1, \dots, v_j\rangle$ the $\mathbb{C}$-subspace of $\mathbb{C}^n$ that they span.
 Let $e_1, e_2, \dots, e_n$ be the standard basis of $\mathbb{C}^n$, and, for $j=1, \dots, n-1$, we set
\[
E_j=\langle e_1,\dots, e_j\rangle,\qquad\hbox{and}\qquad E^{\textrm{opp}}_j=\langle e_n, e_{n-1}, \dots, e_{n-j+1}\rangle.
\]

Write $V_\bullet=\{V_0\subset V_1\subset\cdots \subset V_n\}$ for a complete flag of subspaces of $\C^n$ with $\dim V_p=p$.

Let $w\in S_n$. The Schubert variety is defined by
\begin{equation}\label{Eqn_SchubertVar}
X_w= \big\{V_\bullet\in\mathcal{F}l_n\mid
\dim(V_p\cap E_q)\geq k_{p,q}, \ 1\leq p,q\leq n \big\},
\end{equation}
where $k_{p,q}=\#\{i\leq p\mid w(i)\leq q\}$.
For $v\in S_n$, the opposite Schubert variety is defined by
\begin{equation*}%\label{Eqn_OppositeSchubertVar}
X^v= \big\{V_\bullet\in\mathcal{F}l_n\mid
\dim\big(V_p\cap E^{\textrm{opp}}_q\big)\geq h_{p,q}, \ 1\leq p,q\leq n \big\},
\end{equation*}
where $h_{p,q}=\#\{i\leq p\mid v(i)\geq n+1-q\}$.

For a pair of permutations $x\leq y\in S_n$, we define the Richardson variety
\[
 X^x_y=X_y\cap X^x.
\]

The relevant result for us is the following.

\begin{Theorem}[{\cite[Theorem 4.2]{wooyong}}]\label{Thm_WY} Let $\Phi$ be an interval pattern embedding of $[x,y]$ into $[v,w]$. Then the affine neighbourhoods of $X_y$ and $X_w$ respectively at $x$ and $v$ are isomorphic up to a~Cartesian product with an affine space.
\end{Theorem}

\begin{Remark}[{\cite[Theorem 1]{Woo}}]
Under the same conditions, we also get an isomorphism of Richardson varieties $X_y^x$ and $X_w^v$.
\end{Remark}

We are now ready to prove Theorem \ref{Thm_Main}.

\begin{proof}[Proof of Theorem \ref{Thm_Main}] By Theorem \ref{Thm_WY}, we only have to show that any iteration of our algorithm produces an interval pattern embedding. We keep the same notation as in the proof of Theorem \ref{thm:asyouwish}.

 By the explicit formula \eqref{eq:prime}, we see that $\Phi=\{t, t+1, \dots, n+t-1\}$ gives a common embedding of $x$ into $x'$ and $y$ into $y'$ such that $x'(j)=y'(j)$ for any $1\leq j<t$.

We are now left to show that $[x,y]\simeq [x',y']$ as posets. By \cite[Lemma 2.1]{wooyong}, it is enough to check that
\[
\ell(y)-\ell(x)=\ell(y')-\ell(x'),
\]
where $\ell(w)=\#\{i<j\mid w(i)>w(j)\}$ for any permutation $w$. This is an immediate consequence of \eqref{eq:prime}, from which we see that
\[\ell(x')=\ell(x)+(t-1)(k-1) \qquad\hbox{and}\qquad\ell(y')=\ell(y)+(t-1)(k-1).\tag*{\qed}\]\renewcommand{\qed}{}
\end{proof}

\section{Some applications} \label{Sec_Applications}

\subsection{Associated varieties of highest weight modules}\label{Sec_BBJ} A detailed account of the constructions and properties we mention in this section can be found, for example, in \cite[Section~1]{tanisaki}.

Let $\rho$ be half the sum of the positive roots of $\mathfrak{sl}_n(\mathbb{C})$. For a permutation $w\in S_n$, we denote by $L_w$ the simple $\mathfrak{sl}_n(\mathbb{C})$-module of highest weight $-w(\rho)-\rho$. Denote by $\mathcal{D}$ the sheaf of complex algebraic linear differential operators on the flag variety $\mathcal{F}l_n$ and by $\mathcal{L}_w$ the $\mathcal{D}$-module corresponding to $L_w$ under the Beilinson--Bernstein correspondence, that is the $IC$-extension of the constant local system $\underline{\mathbb{C}}_{C_w}$ of the cell $C_w$. Its characteristic variety $\tt{Ch}(\mathcal{L}_w)$ is a subvariety of the cotangent bundle $T^*\mathcal{F}l_n$ and the corresponding characteristic cycle is a~$\mathbb{Z}_{\geq 0}$-linear combination of the classes of the (closures) of conormal bundles of the Schubert cells:
\[
CC(\mathcal{L}_y)=\sum_{x\leq y}m_{x,y}[\overline{T^*_x\mathcal{F}l_n}].
\]
Determining the numbers $m_{x,y}$ is a natural (but in general very hard) question. As $m_{y,y}=1$, the characteristic variety $\tt{Ch}(\mathcal{L}_y)$ is irreducible if and only if $m_{x,y}=0$ for any $x\neq y$.
The $m_{x,y}$'s only depend on the singularity type of the Schubert variety $X_y$ at $x$ (see \cite[Section~3.2]{vilonenwilliamson}). Thus, we have the following corollary of Theorem~\ref{Thm_Main}:
\begin{Corollary}\label{Cor:BBJ}
Let $x$, $y$, $v$, $w$ be as in Theorem~{\rm \ref{Thm_Main}}. Then $m_{x,y}=m_{v,w}$.
\end{Corollary}

The associated variety $V(L_w)$ can be obtained as the image of $\tt{Ch}(\mathcal{L}_w)$ under the moment map $\gamma\colon T^*\mathcal{F}l_n\rightarrow \mathfrak{g}^*$. This variety $V(L_w)$ is irreducible if and only if $m_{x,y}=0$ for any $x\neq y$ such that $x$ and $y$ lie in the same two-sided Kazhdan--Lusztig cell (see \cite[Introduction]{williamson} for a~sketch of a proof of this fact).

While several examples of reducible characteristic varieties have been known for a long time, it was conjectured that the associated varieties for $\mathfrak{sl}_n(\mathbb{C})$-highest weight modules were irreducible (cf.\ \cite[Conjecture~4.5]{BB} and \cite[Section~10.2]{J}). The first example of a reducible associated variety was exhibited by Williamson in~\cite{williamson}.
This was found by a computer search motivated by constraints that must be satisfied in order for torsion to occur in the intersection cohomology of Schubert varieties.

Williamson's example ended up producing a singularity in the $S_{12}$ flag variety where~$x$ and~$y$ belong to in the same right cell and $m_{x,y}\neq 0$. It was then shown that this singularity is smoothly equivalent to the one studied by Kashiwara and Saito~\cite{kashiwarasaito} where the corresponding characteristic variety was shown to be reducible. Since the $m_{x,y}$'s only depend on the singularity type of the Schubert variety $X_y$ at $x$, this yields a reducible associated variety.
Corollary~\ref{Cor:BBJ} gives a simpler and systematic method for producing similar examples. See Example~\ref{eg:ks}.

\begin{Remark}By \cite[Section~3.1(A)]{tanisaki}, the irreducibility of an associated variety is equivalent to the coincidence of two bases (the \emph{Goldie rank} basis and the \emph{Springer basis}) of a complex irreducible $S_n$-representation, so that our algorithm also provides a method to determine examples of representations for which the two bases differ. A similar question is addressed in Section~\ref{sec:specht}, where we explain how to apply Theorem~\ref{Thm_Main} to the comparison of bases of Specht modules for the Hecke algebra of a symmetric group.
\end{Remark}

\subsection[The 0-1 conjecture and edges of W-graphs for symmetric groups]{The 0-1 conjecture and edges of $\boldsymbol{W}$-graphs for symmetric groups}\label{Sec_01}

Let $x,y\in S_n$.
Let $\mu_{x,y}$ be the coefficient of $q^{(\ell(y)-\ell(x)-1)/2}$ in the Kazhdan--Lusztig polynomial $P_{x,y}(q)$, where $\ell(z)$ denotes the number of inversions of $z\in S_n$. This is the highest possible monomial that can occur in $P_{x,y}(q)$.
The \emph{$0$-$1$ conjecture} stated that $\mu_{x,y}\in \{0,1\}$. This conjecture was disproven in~\cite{mclarnanwarrington} by producing a pair of permutations $x,y\in S_{10}$ with $\mu_{x,y}=4$ via combinatorial methods. We note here that in their example $x$ and $y$ belong to different left cells.

A weaker hope was that the 0-1 conjecture was valid once restricted to pairs of permutations lying in the same left cell.

The relevance of such a conjecture is in the representation theory of the symmetric group Hecke algebra. Indeed, in \cite{KL79} Kazhdan and Lusztig associated with any left cell a representation which could be read off from its $W$-graph, which is a graph having as set of vertices the elements of the cell and edges between any pair $x<y$ labelled by $\mu_{x,y}$. The 0-1 conjecture would have implied that these labels are unnecessary. A counterexample to this weaker version of the 0-1 conjecture was exhibited in \cite{mclarnanwarrington} and relied on computer computations.

By \cite[Corollary 6.3]{wooyong}, if $x$, $y$, $v$, $w$ are as in Theorem \ref{Thm_Main}, then $P_{x,y}(q)=P_{v,w}(q)$, so that in particular we have the following corollary:
\begin{Corollary}\label{Cor_01}
If $x$, $y$, $v$, $w$ are as in Theorem~{\rm \ref{Thm_Main}}, then $\mu_{x,y}=\mu_{v,w}$.
\end{Corollary}
It is hence enough to find a pair of permutations $x$, $y$ not necessarily in the same cell with $\mu_{x,y}>1$ to produce a pair $v^{-1}$, $w^{-1}$ of elements belonging to the same left cell and satisfying $\mu_{v^{-1},w^{-1}}>1$ (as $P_{v,w}=P_{v^{-1},w^{-1}}$).
Since $v$ and $w$ are in the same right cell, $v\inv$ and $w\inv$ are in the same left cell.
Therefore, the failure of the weaker conjecture can be now deduced by combining McLarnan--Warrington's counterexample to the 0-1 conjecture together with Corollary~\ref{Cor_01}. As an illustration, we compute explicitly the obtained counterexample to the weaker conjecture.

Let us identify $S_{10}$ with the permutation of the set $\{0,1,\dots, 9\}$. The second part of \cite[Theorem~1]{mclarnanwarrington} says that if
\[
x=[4321098765],\qquad y=[9467182350]
\]
then $\mu_{x,y}=4$.
The algorithm in the proof of Theorem \ref{thm:asyouwish} applied to $x$, $y$ outputs the pair
\begin{gather*}
 v=[nopqrhijklcdef7845296310smgba],\\
 w=[nopqrhijklcdef78452s9bg1m36a0],
\end{gather*}
where we identified $S_{29}$ with the permutations of $\{0,\dots, 9, a,b,\dots, s\}$.

The pair $v^{-1}$, $w^{-1}$ is an instance of two permutations lying in the same left cell and having $\mu_{v^{-1},w^{-1}}=\mu_{x^{-1},y^{-1}}=\mu_{x,y}=4>1$. Thanks to our main result, this pair can be obtained more conceptually. We want to point out that our counterexample does not appear in the smallest possible rank. In fact the counterexample provided in the second part of \cite[Theorem~1]{mclarnanwarrington} lies in~$S_{16}$ (and has $\mu_{v,w}=5$).

\subsection{Bases of Specht modules}\label{sec:specht}

Consider the Hecke algebra $H_n(q)$ of the symmetric group $S_n$. Let $\{H_x\}_{x\in S_n}$ be the Kazhdan--Lusztig basis of $H_n(q)$ as a $\mathbb{Z}\big[q^{\pm \frac{1}{2}}\big]$-module.
This is a cellular basis (cf.\ \cite[Example~1.2]{grahamlehrer}). Let~$\la$ be a partition and~$Q$ a standard tableau of shape~$\la$. Since the Kazhdan--Lusztig basis is cellular, the set
\[
\{H_x\mid Q(x)=Q\}
\]
is a basis of the Specht module $S^\la$ which does not depend on the choice of $Q$. We call this the KL basis of $S^\lambda$.
Jensen \cite{thorge} studies the analogous situation for the $p$-Kazhdan--Lusztig basis $\{^pH_x\}_{x\in S_n}$ of~$H_n(q)$~-- defined in terms of parity sheaves on the flag variety. Here $p$ is a prime. One defines left, right and two-sided $p$-cells analogously to the case of the Kazhdan--Lusztig basis. Jensen \cite[Theorem~4.33]{thorge} shows that these $p$-cells are the same as the Kazhdan--Lusztig cells. Furthermore, in \cite[Corollary~4.39]{thorge}, he shows that
\[
\{^pH_x\mid Q(x)=Q\}
\]
is a basis of the Specht module $S^\la$ which does not depend on the choice of $Q$. We call this the $p$-KL basis of~$S^\lambda$.\footnote{After this was written, Jensen \cite{jensencellular} further proved that the $p$-Kazhdan Lusztig basis is cellular.}

Define the transition matrix $(^pm_{xy})_{x,y\in S_n}$ by
\[
^pH_y = \sum_x {}^pm_{xy} H_x.
\]
By Theorem \ref{Thm_WY} and \cite[Theorem 3.7]{williamson14}, if the pair $(x,y)$ is such that $x\neq y$ and $x$ is a maximal (w.r.t.\ the Bruhat order) such that ${}^pm_{xy}\neq 0$ then the pair $(v,w)$ has the same property for any interval pattern embedding of $[x,y]$ into $[v,w]$.
Thus, we obtain a third immediate application of Theorem~\ref{Thm_Main}.

\begin{Corollary}\label{Cor_Specht}
If $x$, $y$, $v$, $w$ are as in Theorem~{\rm \ref{Thm_Main}} and $x$ is a maximal element $\neq y$ such that ${}^pm_{xy}\neq 0$ then ${}^pm_{v,w}\neq 0$.
\end{Corollary}
If we can find $x$ and $y$ with $x\neq y$, $Q(x)=Q(y)$ and $^pm_{xy}\neq 0$, then the KL and $p$-KL bases of the corresponding Specht module will disagree. Thus, Corollary \ref{Cor_Specht} is designed to provide such examples, whenever we can find any $x\neq y$ maximal with the property that $^pm_{xy}\neq 0$. For instance, in Example \ref{eg:ks}, $^2m_{v,w}\neq 0$ \cite[Section~5.6]{thorgegeordie}, and $P(v)=P(w)$ have shape $\lambda=(4,4,2,2)$. Hence the KL basis of $S^{(4,4,2,2)}$ disagrees with the 2-KL basis.

For instance, transferring the torsion explosion examples from \cite{explosion} and applying Proposition~\ref{prop:bound}, we get infinitely many examples where these bases differ with $n<A(\log p)^2$ for some constant $A$ (we have made no effort to optimise this bound).

By transferring the examples of non-perverse parity sheaves from \cite{parity}, we see that the change of basis matrix within a Specht module can contain polynomials in $q$ of arbitrarily large degree.

The $W$-graph for the KL basis of a Specht module is bipartite. Since
Nguyen \cite[Theorem~9.8]{nguyen} shows that all strongly connected admissible $W$-graphs in type $A$ are the Kazhdan--Lusztig ones, this implies that whenever the $p$-KL basis differs from the KL basis, the corresponding $W$-graph for the $p$-KL basis is not bipartite. An explicit example of a non-bipartite $W$-graph for $p=2$ in type $C_3$ was previously exhibited in~\cite{thorge}.

\subsection*{Acknowledgements}
The authors would like to thank the Institut Henri Poincar\'e in Paris, and the organisers of the ``Representation Theory'' Trimester. M.L.~acknowledges the MIUR Excellence Department Project awarded to the Department of Mathematics, University of Rome Tor Vergata, CUP E83C18000100006, and the PRIN2017 CUP E8419000480006. P.M.~acknowledges support from ARC grants DE150101415 and DP180103150.
We thank G.~Williamson for useful conversations and the anonymous referees for their valuable input.

\pdfbookmark[1]{References}{ref}
\LastPageEnding


\begin{thebibliography}{99}
\footnotesize\itemsep=0pt

\bibitem{patternavoidance}
Abe H., Billey S., Consequences of the {L}akshmibai--{S}andhya theorem: the
 ubiquity of permutation patterns in {S}chubert calculus and related geometry,
 in Schubert Calculus~-- {O}saka 2012, \textit{Adv. Stud. Pure Math.},
 Vol.~71, Editors H.~Naruse, T.~Ikeda, M.~Masuda, T.~Tanisaki, \href{https://doi.org/10.2969/aspm/07110001}{Math. Soc.
 Japan}, Tokyo, 2016, 1--52, \href{https://arxiv.org/abs/1403.4345}{arXiv:1403.4345}.

\bibitem{ariki}
Ariki S., Robinson--{S}chensted correspondence and left cells, in Combinatorial
 Methods in Representation Theory ({K}yoto, 1998), \textit{Adv. Stud. Pure
 Math.}, Vol.~28, \href{https://doi.org/10.2969/aspm/02810001}{Kinokuniya}, Tokyo, 2000, 1--20, \href{https://arxiv.org/abs/math.QA/9910117}{arXiv:math.QA/9910117}.

\bibitem{BB}
Borho W., Brylinski J.-L., Differential operators on homogeneous spaces. {III}.
 {C}haracteristic varieties of {H}arish-{C}handra modules and of primitive
 ideals, \href{https://doi.org/10.1007/BF01388547}{\textit{Invent. Math.}} \textbf{80} (1985), 1--68.

\bibitem{BCP}
Bosma W., Cannon J., Playoust C., The {M}agma algebra system. {I}.~{T}he user
 language, \href{https://doi.org/10.1006/jsco.1996.0125}{\textit{J.~Symbolic Comput.}} \textbf{24} (1997), 235--265.

\bibitem{gm}
Garsia A.M., McLarnan T.J., Relations between {Y}oung's natural and the
 {K}azhdan--{L}usztig representations of {$S_n$}, \href{https://doi.org/10.1016/0001-8708(88)90060-6}{\textit{Adv. Math.}}
 \textbf{69} (1988), 32--92.

\bibitem{grahamlehrer}
Graham J.J., Lehrer G.I., Cellular algebras, \href{https://doi.org/10.1007/BF01232365}{\textit{Invent. Math.}}
 \textbf{123} (1996), 1--34.

\bibitem{HN}
Howlett R., Nguyen V.M., W-graph magma programs, {a}vailable at
 \url{http://www.maths.usyd.edu.au/u/bobh/magma/}, 2013.

\bibitem{thorge}
Jensen L.T., The {ABC} of {$p$}-cells, \href{https://doi.org/10.1007/s00029-020-0552-1}{\textit{Selecta Math. (N.S.)}}
 \textbf{26} (2020), 28, 46~pages, \href{https://arxiv.org/abs/1901.02323}{arXiv:1901.02323}.

\bibitem{jensencellular}
Jensen L.T., Cellularity of the $p$-canonical basis for symmetric groups,
 \href{https://arxiv.org/abs/2009.11715}{arXiv:2009.11715}.

\bibitem{thorgegeordie}
Jensen L.T., Williamson G., The {$p$}-canonical basis for {H}ecke algebras, in
 Categorification and Higher Representation Theory, \textit{Contemp. Math.},
 Vol.~683, \href{https://doi.org/10.1090/conm/683}{Amer. Math. Soc.}, Providence, RI, 2017, 333--361,
 \href{https://arxiv.org/abs/1510.01556}{arXiv:1510.01556}.

\bibitem{J}
Joseph A., On the variety of a highest weight module, \href{https://doi.org/10.1016/0021-8693(84)90100-5}{\textit{J.~Algebra}}
 \textbf{88} (1984), 238--278.

\bibitem{kashiwarasaito}
Kashiwara M., Saito Y., Geometric construction of crystal bases, \href{https://doi.org/10.1215/S0012-7094-97-08902-X}{\textit{Duke
 Math.~J.}} \textbf{89} (1997), 9--36, \href{https://arxiv.org/abs/q-alg/9606009}{arXiv:q-alg/9606009}.

\bibitem{KL79}
Kazhdan D., Lusztig G., Representations of {C}oxeter groups and {H}ecke
 algebras, \href{https://doi.org/10.1007/BF01390031}{\textit{Invent. Math.}} \textbf{53} (1979), 165--184.

\bibitem{lakshmibaisandhya}
Lakshmibai V., Sandhya B., Criterion for smoothness of {S}chubert varieties in
 {${\rm Sl}(n)/B$}, \href{https://doi.org/10.1007/BF02881113}{\textit{Proc. Indian Acad. Sci. Math. Sci.}} \textbf{100}
 (1990), 45--52.

\bibitem{mclarnanwarrington}
McLarnan T.J., Warrington G.S., Counterexamples to the 0-1 conjecture,
 \href{https://doi.org/10.1090/S1088-4165-03-00178-X}{\textit{Represent. Theory}} \textbf{7} (2003), 181--195,
 \href{https://arxiv.org/abs/math.CO/0209221}{arXiv:math.CO/0209221}.

\bibitem{parity}
McNamara P.J., Non-perverse parity sheaves on the flag variety,
 \href{https://arxiv.org/abs/1812.00178}{arXiv:1812.00178}.

\bibitem{nguyen}
Nguyen V.M., Type {$A$} admissible cells are {K}azhdan--{L}usztig,
 \href{https://doi.org/10.5802/alco.91}{\textit{Algebr. Comb.}} \textbf{3} (2020), 55--105, \href{https://arxiv.org/abs/1807.07457}{arXiv:1807.07457}.

\bibitem{stanley}
Stanley R.P., Enumerative combinatorics, {V}ol.~2, \textit{Cambridge Studies in
 Advanced Mathematics}, Vol.~62, \href{https://doi.org/10.1017/CBO9780511609589}{Cambridge University Press}, Cambridge, 1999.

\bibitem{tanisaki}
Tanisaki T., Characteristic varieties of highest weight modules and primitive
 quotients, in Representations of {L}ie Groups, {K}yoto, {H}iroshima, 1986,
 \textit{Adv. Stud. Pure Math.}, Vol.~14, \href{https://doi.org/10.2969/aspm/01410001}{Academic Press}, Boston, MA, 1988,
 1--30.

\bibitem{vilonenwilliamson}
Vilonen K., Williamson G., Characteristic cycles and decomposition numbers,
 \href{https://doi.org/10.4310/MRL.2013.v20.n2.a11}{\textit{Math. Res. Lett.}} \textbf{20} (2013), 359--366, \href{https://arxiv.org/abs/1208.1198}{arXiv:1208.1198}.

\bibitem{williamson14}
Williamson G., On an analogue of the {J}ames conjecture, \href{https://doi.org/10.1090/S1088-4165-2014-00447-3}{\textit{Represent.
 Theory}} \textbf{18} (2014), 15--27, \href{https://arxiv.org/abs/1212.0794}{arXiv:1212.0794}.

\bibitem{williamson}
Williamson G., A reducible characteristic variety in type {$A$}, in
 Representations of Reductive Groups, \textit{Progr. Math.}, Vol.~312,
 \href{https://doi.org/10.1007/978-3-319-23443-4_19}{Birkh\"auser/Springer}, Cham, 2015, 517--532, \href{https://arxiv.org/abs/1405.3479}{arXiv:1405.3479}.

\bibitem{explosion}
Williamson G., Schubert calculus and torsion explosion (with a joint appendix
 with Alex Kontorovich and Peter J.~McNamara), \href{https://doi.org/10.1090/jams/868}{\textit{J.~Amer. Math. Soc.}}
 \textbf{30} (2017), 1023--1046, \href{https://arxiv.org/abs/1309.5055}{arXiv:1309.5055}.

\bibitem{Woo}
Woo A., Interval pattern avoidance for arbitrary root systems, \href{https://doi.org/10.4153/CMB-2010-080-2}{\textit{Canad.
 Math. Bull.}} \textbf{53} (2010), 757--762, \href{https://arxiv.org/abs/math.CO/0611328}{arXiv:math.CO/0611328}.

\bibitem{wooyong}
Woo A., Yong A., Governing singularities of {S}chubert varieties,
 \href{https://doi.org/10.1016/j.jalgebra.2007.12.016}{\textit{J.~Algebra}} \textbf{320} (2008), 495--520, \href{https://arxiv.org/abs/math.AG/0603273}{arXiv:math.AG/0603273}.

\end{thebibliography}
\end{document}